\documentclass[12pt]{elsarticle}

\usepackage[english]{babel}
\usepackage{graphicx,epstopdf,epsfig,kpfonts}
\usepackage{amsfonts,epsfig,fancyhdr,graphics,amsmath,amssymb,float,subfig}
\usepackage{enumerate,mathtools}
\usepackage[matrix,arrow,tips,curve]{xy}
\usepackage{latexsym,eucal,color}
\usepackage{pstricks}
\usepackage{dsfont}
\usepackage{setspace}
\usepackage{multicol}
\usepackage{enumitem}
\usepackage{tikz,pgfplots,graphicx}
\usetikzlibrary{patterns}
\usetikzlibrary{shapes,arrows,positioning,decorations.markings,math}
\usepackage{tkz-euclide}
\pgfplotsset{compat=1.11}

\usepackage[hyphens]{xurl}
\usepackage[colorlinks=true, breaklinks=true, linkcolor=blue, urlcolor=blue, citecolor=black]{hyperref}
\usepackage{xcolor}
\definecolor{Maroon}{RGB}{128,0,0}
\definecolor{Blue}{RGB}{0,0,255}

\newcommand{\defin}[1]{\textcolor{Maroon}{\emph{#1}}\index{#1}}
\usepackage{amsthm}
\usepackage{cleveref}
\usepackage{lipsum}

\usepackage{atbegshi}
\newcommand{\handlethispage}{}
\AtBeginShipout{\handlethispage}
\theoremstyle{plain}
\numberwithin{equation}{section}
\newtheorem{thm}{Theorem}
\newtheorem{lemma}[thm]{Lemma}
\newtheorem{prop}[thm]{Proposition}
\newtheorem{corollary}[thm]{Corollary}

\newtheorem{definition}{Definition}
\newtheorem{notation}{Notation}

\theoremstyle{definition}
\newtheorem{remark}{Remark}
\newtheorem{problem}{Problem}

\newcommand{\dd}{\mathsf{deg}}

\newenvironment{txteq*}
{
	\begin{equation*}
	\begin{minipage}[t]{0.85\textwidth} 
	\em                                
}

\title{Spectral Methods for Matrix Product Factorization}

\author[SA]{Saieed Akbari}
\author[YF]{Yi-Zheng Fan}
\author[YF]{Fu-Tao Hu}
\author[BM]{Babak Miraftab}
\author[YF]{Yi Wang}

\address[SA]{Department of Mathematics, Sharif University of Technology, Tehran, Iran}
\address[YF]{School of Mathematical Sciences (SMS), Anhui University, China}
\address[BM]{School of Computer Science, Carleton University, Ottawa, Canada}

\ead{s_akbari@sharif.edu}
\ead{fanyz@ahu.edu.cn}
\ead{hufu@ahu.edu.cn}
\ead{babak.miraftab@carleton.ca}

\begin{document}
\bibliographystyle{LAA}


\begin{abstract}
A graph $G$ is factored into graphs $H$ and $K$ via a matrix product if there exist adjacency matrices $A$, $B$, and $C$ of $G$, $H$, and $K$, respectively, such that $A = BC$. In this paper, we study the spectral aspects of the matrix product of graphs, including regularity, bipartiteness, and connectivity.
We show that if a graph $G$ is factored into a connected graph $H$ and a graph $K$ with no isolated vertices, then certain properties hold. 
If $H$ is non-bipartite, then $G$ is connected. If $H$ is bipartite and $G$ is not connected, then $K$ is a regular bipartite graph, and consequently, $n$ is even. 
Furthermore, we show that trees are not factorizable, which answers a question posed by Maghsoudi et al.
\end{abstract}

\begin{keyword}
Matrix product \sep eigenvalues \sep $(0,1)$-matrices
\MSC[2010]{ 05C50, 15A18 }
\end{keyword}

\maketitle

\section{Introduction}\label{sec:intro}
In this paper, we only deal with simple graphs, which are graphs without loops or multiedges.
A graph has a \defin{factorization} with respect to a given graph product if it can be represented as the product of two graphs.
For the Cartesian product, this means that a simple graph $G$ can be expressed as $H\Box K$.
The factorization problem investigates which graph classes admit a  factorization. This problem has been studied for various graph products such as the Cartesian product, Tensor Product, Strong product, etc. For a comprehensive survey, we refer readers to \cite{MR2817074}.
In this paper, we delve into the factorization problem with respect to the \defin{matrix product of graphs}. Prasad et al. introduced the concept of the matrix product of graphs in \cite{Manjunatha}.
A graph $G$ has a factorization into graphs $H$ and $K$ if there are adjacency matrices of $A, B$, and $C$ of $G,H$ and $K$, respectively such that $A=BC$.
For instance, the cycle graph on $6$ vertices i.e. $C_6$ has a  factorization into the following graphs:
\begin{figure}[H]
    \centering
    \includegraphics[scale=0.7]{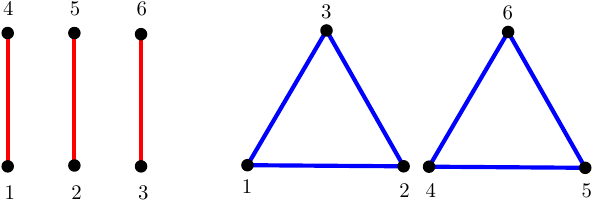}
    \caption{A matching of size 3 and union of two triangles}
    \label{fig:enter-label}
\end{figure}
\noindent Because one can see that
\begin{align*}
    \begin{bmatrix}
	0 & 1 & 1 & 0 & 0 & 0 \\
	1 & 0 & 1 & 0 & 0 & 0 \\
	1 & 1 & 0 & 0 & 0 & 0 \\
	0 & 0 & 0 & 0 & 1 & 1 \\
    0 & 0 & 0 & 1 & 0 & 1 \\
	0 & 0 & 0 & 1 & 1 & 0 \\
	\end{bmatrix} \begin{bmatrix}
	0 & 0 & 0 & 1 & 0 & 0\\
	0 & 0 & 0 & 0 & 1 & 0\\
	0 & 0 & 0 & 0 & 0 & 1\\
	1 & 0 & 0 & 0 & 0 & 0\\
    0 & 1 & 0 & 0 & 0 & 0\\
	0 & 0 & 1 & 0 & 0 & 0\\
	\end{bmatrix} = \begin{bmatrix}
	0 & 0 & 0 & 0 & 1 & 1\\
    0 & 0 & 0 & 1 & 0 & 1\\
    0 & 0 & 0 & 1 & 1 & 0\\
    0 & 1 & 1 & 0 & 0 & 0\\
    1 & 0 & 1 & 0 & 0 & 0\\
    1 & 1 & 0 & 0 & 0 & 0\\
	\end{bmatrix}.
 \end{align*}

Recently, the factorization problem with respect to the matrix product has garnered attention from researchers.
For example, Prasad et al. proved in \cite{Manjunatha} that if a graph $G$ has a factorization, then it must have an even number of edges, see \cite[Theorem 4]{Manjunatha}.
Additionally, in \cite{maghsoudi2023matrix}, Maghsoudi et al. characterized which complete graphs and complete bipartite graphs have factorizations, see \cite[Theorem 1.1 and Theorem 1.2]{maghsoudi2023matrix}.
In the same paper, the authors propose the question of whether trees are factorizable or not.
It is observed that any tree with an even number of vertices cannot be factorizable, given that it contains an odd number of edges.
In this paper, we show that every forest with no isolated vertices and with an odd number of connected components is not factorizable  and moreover we show that every tree of order at least 2 is not factorizable.

\noindent {\bf An overview of our paper:} Our paper is organized as follows. After presenting the preliminaries, we provide some general spectral results about factorization in \cref{gen}. We then continue in \cref{reg_sec} with results regarding the regularity of factorizable graphs. Following this, we discuss the bipartiteness of factorizable graphs. \Cref{conn_sec} is devoted to the connectivity of factorizable graphs. Finally, we close the paper with our main result in \cref{tree_sec}, which focuses on the factorization of jungles and trees. 


\section{Preliminaries}\label{sec:prem}
\noindent Throughout this paper, we consistently assume that $G$ denotes a simple graph, which implies that $G$ is free of loops and multiple edges.
Let $G$ be a graph. Then $V(G)$ denotes the vertex set of $G$, and $E(G)$ denotes its edge set.
A \defin{positive} matrix denoted by $A>0$, is a matrix in which all the entries are greater than zero.
The \defin{adjacency matrix} $A=[a_{ij}]$  of a graph $G $ of order $n$ is defined as follows:

\[
a_{ij} =
\begin{cases}
1, & \text{if there is an edge between vertex } i \text{ and vertex } j, \\
0, & \text{otherwise}.
\end{cases}
\]
The \defin{eigenvalues} of a graph are derived from its adjacency matrix. The eigenvalues provide important structural information about the graph.
\begin{definition}
We say a graph $G$ is \defin{factored} into graphs $H$ and $K$ if there exist adjacency matrices $A$, $B$, and $C$ of $G$, $H$, and $K$, respectively such that $A = BC$.
In this case, $H$ and $K$ are called \defin{factors} of $G$ and we say $G$ is \defin{factorizable}.
\end{definition}

\begin{remark}\label{diff_adj_matrix}
We note that if a graph $G$ can be factored into graphs $H$ and $K$, then any graph isomorphic to $G$ can also be factored into two graphs $H'$ and $K'$, where $H' \cong H$ and $K' \cong K$.
This is because the adjacency matrices $A$ and $B$ represent isomorphic graphs if and only if there exists a permutation matrix $P$ such that $B = P^TAP$. Since permutation matrices are orthogonal ($P^T = P^{-1}$), this implies that the matrices $A$ and $B$ are similar.
\end{remark}


\section{General Results on Matrix Products}\label{gen}

\begin{lemma}{\rm\cite[Section 6.5]{MR0276251}}\label{sim_diag}
Diagonalizable matrices are simultaneously diagonalizable if and only if they are commutative.
\end{lemma}

\noindent Let $G$ be factored into graphs $H$ and $K$.
Then if $\lambda$ is an eigenvalue of $G$, then there are eigenvalues $\mu$ and $\gamma$ of $H$ and $K$ respectively such that $\lambda=\mu\gamma$.
In the following theorem, we show that if a connected graph $G$ is factored into graphs $H$ and $K$, then the largest eigenvalue of $G$ is the product of the largest eigenvalues of $H$ and $K$.

\begin{thm}\label{eigen_max}
Let a connected graph $G$ be factored into graphs $H$ and $K$.
Then $$\lambda_{\text {max}}(G)=\lambda_{\text {max}}(H) \lambda_{\text {max}}(K).$$
\end{thm}

\begin{proof}
Let $A$, $B$, and $C$ represent the adjacency matrices of graphs $G$, $H$, and $K$, respectively, such that $A=BC$.
It is clear that $A$, $B$, and $C$ not only commute with each other but are also symmetric matrices.
We conclude that they can be diagonalized simultaneously, as discussed in \Cref{sim_diag}.
From this simultaneous diagonalization, we infer that the matrices $A$, $B$, and $C$ share a common Perron vector $u$.
Since $u$ is an eigenvector for $B$ and $u>0$, we have $Bu=\lambda_{\text {max}}(B)u$.
Similarly $Cu=\lambda_{\text {max}}(C)u$.
Thus $\lambda_{\text {max}}(A)=\lambda_{\text {max}}(B)\lambda_{\text {max}}(C)$.
\end{proof}

\noindent \Cref{eigen_max} can be improved as follows:

\noindent Let a connected graph $G$ be factored into graphs $H$ and $K$.
If one of $G,H,K$ is connected, then
$$\lambda_{\text {max}}(G)=\lambda_{\text {max}}(H) \lambda_{\text {max}}(K),$$
each component of $G, H, K$ has the same spectral radius as $G,H,K$ respectively,
and hence none of $G,H,K$ contains isolated vertices.

The proof of the equality is similar by the fact that $A,B,C$ share a common Perron vector $u >0$.
If one of $G, H, K$ has a component, say $G_0$ of $G$, then
$$ A(G_0) u|_{V(G_0)} = \lambda_{\text {max}}(G)u|_{V(G_0)},$$
where $u|_{V(G_0)}$ is a subvector of $u$ indexed by $V(G_0)$, which is positive.
So $\lambda_{\text {max}}(G)=\lambda_{\text {max}}(G_0)$;
and $G_0$ is not an isolated vertices, otherwise $\lambda_{\text {max}}(G_0)=0$; a contradiction.

\begin{remark}
If $G$ is not connected, then \Cref{eigen_max} does not hold by the following example.
The matrix product of graphs in \Cref{fig:enter-label} can be generalized to the following product,
namely,
$$ C_{2n}= (n K_2) * (2 C_n),$$
where $n K_2$ denotes an $n$-matching consisting of $n$ disjoint edges $\{1,n+1\}, \{2,n+2\}, \ldots, \{n,2n\}$, and $2 C_n$ denotes the union of two disjoint cycles $C_n$, which have edges $\{1,2\}, \{2,3\},\ldots,\{n-1,n\}, \{n,1\}$,  and $\{n+1,n+2\}, \ldots, \{2n-1,2n\}, \{2n,n+1\}$ respectively.

Let $H= (n K_2) \uplus (2 C_n)$ and $K=(2C_n) \uplus (nK_2)$, where $\uplus$ denotes the disjoint union of two graphs.
We have
$$ (2 C_{2n})= H * K,$$
but
$$2=\lambda_{\text {max}}(2 C_{2n}) \ne \lambda_{\text {max}}(H) \times \lambda_{\text {max}}(K)=2 \times 2=4.$$
\end{remark}

\noindent Let $G$ be a graph, $u,v\in V(G)$. Then, we say that $u$ and $v$ are \defin{connected} in $G$ if there exists a path that starts at $u$ and ends at $v$.

\begin{lemma}{\rm\cite[Corollary 2]{Manjunatha}}\label{deg}
Let $G$ be factored into graphs $H$ and $K$.
If vertices $u$ and $v$ are connected in $H$, then $\dd_K(u)={\dd}_K(v)$.
\end{lemma}

\begin{lemma}{\rm\cite[Theorem 7]{Manjunatha}}\label{degree}
Let $G$ be factored into graphs $H$ and $K$.
Then, the following holds:
\begin{align*}
\dd_{G}(v) = \dd_{H}(v)\dd_{K}(v).
\end{align*}
\end{lemma}

\begin{lemma}{\rm \cite[Corollary 1.2]{ilic}}\label{ilic}
Let $G$ be a connected graph with the degree sequence $(d_1, \ldots, d_n)$. Then
$$\sum_{i=1}^n d_i^2\leq \frac{(n+1)^2|E(G)|^2}{2n(n-1)}$$
\end{lemma}

\begin{prop}
Let a graph $G$ be factored into  $H$ and $K$ which both are without isolated vertices.
Then the following holds:

\begin{enumerate}
\item $|E(H)|$ or $|E(K)|$ is not greater than $|E(G)|$.
\item $|E(G)|\leq \min\{\Delta(H)|E(K)|,\Delta(K)|E(H)|\}$.
\item If $|V(G)|>4$, then $|E(G)|\leq \frac{1}{2}|E(H)||E(K)|$.
\end{enumerate}
\end{prop}

\begin{proof}
It follows from \Cref{degree} that $d_i=d_i'd_i''$, where $d_i=\dd_G(v_i)$, $d_i'=\dd_H(v_i)$ and $d_i''=\dd_K(v_i)$.
For the first item, we have the following calculation:
$$
2|E(G)|= \sum_{i=1}^n d_i =\sum_{i=1}^n d_i^{\prime} d_i^{\prime \prime} \geq \sum_{i=1}^n d_i^{\prime \prime}=2|E(H)|
$$
For the second item, one can see that $$2|E(G)|=\sum_{i=1}^n d_i= \sum_{i=1}^n d_i'd_i''\leq \sum_{i=1}^n \Delta(H)d_i''=2\Delta(H)|E(K)|.$$
Analogously one can show that $|E(G)|\leq \Delta(K)|E(H)|$.\\
Next, we apply the Cauchy-Schwarz inequality to \(2|E(G)| = \sum_{i=1}^{n} d_i'd_i''\), and we obtain \(2|E(G)| \leq \sqrt{\sum_{i=1}^{n} d_i'^2}\sqrt{\sum_{i=1}^{n} d_i''^2}\).
Now, it follows from \Cref{ilic} that
\[
\sum_{i=1}^n d_i'^2 \leq \frac{(n+1)^2|E(H)|^2}{2n(n-1)}\text{ and } \sum_{i=1}^n d_i''^2 \leq \frac{(n+1)^2|E(K)|^2}{2n(n-1)}.
\]
One can easily see that if \(n \geq 5\), then \((n+1)^2 \leq 2n(n-1)\) which implies that $2|E(G)|\leq |E(H)||E(K)|$, as desired.
\qedhere
\end{proof}


\section{Regularity of Matrix Products}\label{reg_sec}

\noindent Considering \Cref{degree} implies that if both $H$ and $K$ are regular graphs, then $G$ is regular too.

\begin{corollary}\label{cor->deg}
Let $G$ be factored into regular graphs $H$ and $K$.
Then $G $ is regular as well.
\end{corollary}

\noindent Next, we establish that the converse of \Cref{cor->deg} holds true as well under the condition of ``connectedness''. However, before delving into the proof, we require the following lemma and notation, proved by Hoffman.
\begin{notation}
We denote a matrix where every entry is equal to one by $J$, and $j$ denotes a vector with entries one.
\end{notation}

\begin{lemma}{\rm\cite[Theorem 2]{hoffman_63}}\label{hoff}
Let $G$ be a graph with the adjacency matrix $A$.
There exists a polynomial $p(x)$ such that $J=p(A)$
if and only if $G$ is regular and connected.
\end{lemma}

\noindent We note that the vector $j$ is an eigenvector for the eigenvalue $k$ in a $k$-regular graph.

\begin{thm}\label{reg}
Let a connected regular graph $G$ be factored into graphs $H$ and $K$.
Then both $H$ and $K$ are regular.
\end{thm}

\begin{proof}
Let $A$, $B$, and $C$ denote the adjacency matrices of $G $, $H$, and $K$, respectively.
Then we have $A=BC$.
Since $G$ is regular, by \Cref{hoff}, there exists a polynomial $p(x) = \sum_{i=0}^{k} a_i x^i$ such that $J = p(A)$, implying that $J = \sum_{i=1}^{k} a_i A^i = \sum_{i=1}^{k} a_i (BC)^i = \sum_{i=1}^{k} a_i B^i C^i$. Then, we proceed with the following calculations:
\begin{align*}
    BJ&=\sum_{i=1}^k a_iB^{i+1}C^i\\
    &=\sum_{i=1}^k a_iB^iC^iB\\
    &=JB
\end{align*}
Since $BJ = JB$, one can see that $H$ is regular. Similarly, one can demonstrate that $K$ is regular.
\end{proof}

\begin{remark}
\Cref{reg} does not hold for non-connected graphs.
Consider $H$ and $K$ with the following labels:
\begin{figure}[H]
    \centering
    \includegraphics[scale=0.8]{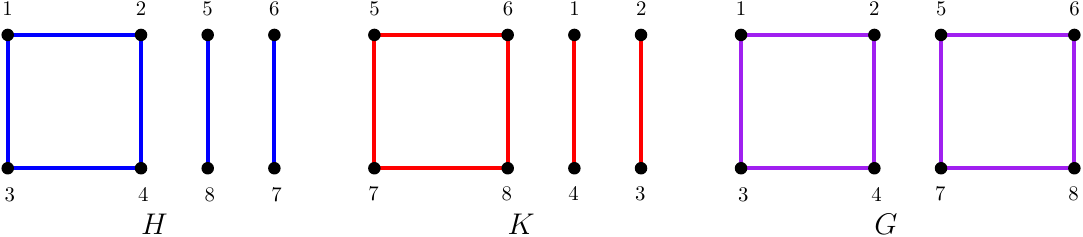}%
    \caption{The matrix product of $H$ and $K$ results the graph $G$.}%
    \label{factor_G}
\end{figure}
\noindent A simple calculation of the adjacency matrices verifies that:
\begin{align*}
    \begin{bmatrix}
	0 & 1 & 1 & 0 & 0 & 0 & 0 & 0 \\
	1 & 0 & 0 & 1 & 0 & 0 & 0 & 0 \\
 	1 & 0 & 0 & 1 & 0 & 0 & 0 & 0 \\
	0 & 1 & 1 & 0 & 0 & 0 & 0 & 0 \\
	0 & 0 & 0 & 0 & 0 & 0 & 0 & 1 \\
    0 & 0 & 0 & 0 & 0 & 0 & 1 & 0 \\
	0 & 0 & 0 & 0 & 0 & 1 & 0 & 0 \\
 	0 & 0 & 0 & 0 & 1 & 0 & 0 & 0 \\
	\end{bmatrix}
    \begin{bmatrix}
	0 & 0 & 0 & 1 & 0 & 0 & 0 & 0 \\
	0 & 0 & 1 & 0 & 0 & 0 & 0 & 0 \\
 	0 & 1 & 0 & 0 & 0 & 0 & 0 & 0 \\
	1 & 0 & 0 & 0 & 0 & 0 & 0 & 0 \\
	0 & 0 & 0 & 0 & 0 & 1 & 1 & 0 \\
    0 & 0 & 0 & 0 & 1 & 0 & 0 & 1 \\
	0 & 0 & 0 & 0 & 1 & 0 & 0 & 1 \\
 	0 & 0 & 0 & 0 & 0 & 1 & 1 & 0 \\
	\end{bmatrix}
 = \begin{bmatrix}
	0 & 1 & 1 & 0 & 0 & 0 & 0 & 0 \\
	1 & 0 & 0 & 1 & 0 & 0 & 0 & 0 \\
 	1 & 0 & 0 & 1 & 0 & 0 & 0 & 0 \\
	0 & 1 & 1 & 0 & 0 & 0 & 0 & 0 \\
	0 & 0 & 0 & 0 & 0 & 1 & 1 & 0 \\
    0 & 0 & 0 & 0 & 1 & 0 & 0 & 1 \\
	0 & 0 & 0 & 0 & 1 & 0 & 0 & 1 \\
 	0 & 0 & 0 & 0 & 0 & 1 & 1 & 0 \\
	\end{bmatrix},
 \end{align*}
\end{remark}


\section{Bipartiteness  of Matrix Products}\label{bipar_sec}

\noindent In this section, we study the properties of bipartite factorizable graphs. First, we review some standard spectral facts about bipartite graphs.

\begin{lemma}{\rm\cite{MR3185588}}\label{bipartite}
If $\lambda_1\geq \cdots\geq \lambda_n$ are the eigenvalues of a graph $G$, then the following statements hold:
\begin{enumerate}
    \item $|\lambda_i|\leq \lambda_1$.
    \item If $\lambda_1=\lambda_2$, then $G$ is not connected.
    \item If $\lambda_1=-\lambda_n$, then one of  components of $G$ is bipartite
    \item $\lambda_i=-\lambda_{n+1-i}$ for $i=1,\ldots,n$ if and only if $G$ is biparitite.
\end{enumerate}
\end{lemma}

\begin{thm}
Let a connected bipartite graph $G$ be factored into $H$ and $K$.
Then by permutation, the adjacency matrices $A,B,C$ of $G,H,K$ respectively hold the following relation:
\begin{equation}\label{prod}
\begin{bmatrix}
O & A_{12} \\
A_{12}^\top & O
\end{bmatrix}
=\begin{bmatrix}
O & B_{12} \\
B_{12}^\top & O
\end{bmatrix}
\begin{bmatrix}
C_{11} & O\\
O & C_{22}
\end{bmatrix}.
\end{equation}
Consequently, one of $H,K$ is bipartite and the other is not connected.
\end{thm}

\begin{proof}
Let $A, B$ and $C$ be the adjacency matrices of $G$, $H$ and  $K$, respectively such that $A=B C$. Assume $A$ has the form in (\ref{prod}) and  $v=\left[\begin{array}{l}X \\ Y\end{array}\right]>0$ is the Perron vector corresponding to $\lambda_{\max }(G)$, where $X$ and $Y$ are the subvectors of $v$ indexed by two parts of $G$ respectively.
It follows from \Cref{bipartite}  that $\lambda_{\text {max }}(G)$ and $-\lambda_{\text {max }}(G)$ are simple and one can see that the vector $\mu=\left[\begin{array}{c}X \\ -Y\end{array}\right]$ is an eigenvector corresponding to the eigenvalue $-\lambda_{\max }(G)$.
It follows from \Cref{eigen_max} that
$\lambda_{\text {max }}(G)=\lambda_{\text {max }}(H) \lambda_{\text {max }}(K)$ and so $-\lambda_{\text {max }}(G)=-\lambda_{\text {max }}(H) \lambda_{\text {max }}(K)$.
Since $A, B$ and $ C$ mutually commute, and they are symmetric, they are simultaneously diagonetizable, and so we deduce that $\mu$ is an eigenvector corresponding to $-\lambda_{\max }(H)$ and $\lambda_{\text {max }}(K)$ or $\lambda_{\text {max }}(H)$ and $-\lambda_{\text {max }}(K)$.

Without loss of generality, assume that the first case occurs.
Let $B$ have a partition conform with $A$, namely,
$$B=\left[\begin{array}{ll}B_{11} & B_{12} \\ B_{12}^\top & B_{22}\end{array}\right].$$
Then one can see that
$$
\begin{array}{ll}
B=\left[\begin{array}{ll}B_{11} & B_{12} \\ B_{12}^\top & B_{22}\end{array}\right]\left[\begin{array}{l}
X \\
Y
\end{array}\right]=\lambda_{\max}{(H)}\left[\begin{array}{l}
X \\
Y
\end{array}\right],\left[\begin{array}{ll}B_{11} & B_{12} \\ B_{12}^\top & B_{22}\end{array}\right]\left[\begin{array}{c}
X \\
-Y
\end{array}\right]=-\lambda_{\text {max}}(H)\left[\begin{array}{c}
X \\
-Y
\end{array}\right]
\end{array}
$$
Equivalently $B_{11} X \pm B_{12} Y= \pm \lambda_{\text {max }}(H) X, ~ B_{12} ^{\top} X \pm B_{12} Y=\lambda_{\text {max }}(H) Y$.
Thus $B_{11} X=0$ and $B_{22}Y=0$ which implies that $B_{11}$ and $B_{22}$ are zero. This means that $B$ has the form in (\ref{prod}) and hence $H$ is bipartite, as desired.

Let $C$ have a partition conform with $A$, namely,
$$B=\left[\begin{array}{ll}C_{11} & C_{12} \\ C_{12}^\top & C_{22}\end{array}\right].$$
Then one can see that
$$ C \nu = \lambda_{\text {max }}(K) \nu, ~ C \mu = \lambda_{\text {max }}(K) \mu.$$
By a similar discussion, we have $C_{12}=0, C_{12}^\top=0$, which implies the desired form of $C$ in (\ref{prod}) and the disconnectedness of $K$.
\end{proof}

\begin{prop}\label{G_or_H}
Let a bipartite graph $G$ be factored into graphs $H$ and $K$.
Then at most one of the factors is connected.
\end{prop}

\begin{proof}
Let us assume, for contradiction, that both $H$ and $K$ are connected.
By \Cref{deg}, we deduce that both $H$ and $K$ are regular.
Suppose $H$ and $K$ are $r$-regular and $s$-regular, respectively.
According to \Cref{degree}, $G$ is also $rs$-regular.
Hence, $rs$ is the largest eigenvalue of $G$.
As by \Cref{bipartite}(4), $-rs$ is also an eigenvalue of $G$.
However, \Cref{bipartite}(2) implies that the multiplicity of $s$ and $r$ is one, leading to a contradiction, as stated in \Cref{bipartite}(1).
\end{proof}


\section{Connectivity  of Matrix Products}\label{conn_sec}

\begin{lemma}{\rm\cite[Lemma 10.3.3]{Agg}}\label{connected}
A graph $G$ of order $n$ with the adjacency matrix $A$ is connected if and only if $\sum_{i=0}^{n-1} A^i>0$.
\end{lemma}

\begin{lemma}\label{lem:A>0}
Let $G$ be a non-bipartite graph with adjacency matrix $A$.
Then $G$ is connected if and only if there exists a positive integer $k$ such that  $A^k>0$.
\end{lemma}
\begin{proof}
    We first assume that $G$ is connected.
    Let $C_r$ be an odd cycle in $G$ and for every vertex $v$, let $P_v$ be a shortest path between $v$ and $C_r$.
    Assume that $u$ and $v$ are two vertices of $G$ such that $V(C_r)\cap V(P_u)=\{x\}$ and $V(C_r)\cap V(P_v)=\{y\}$.
    Let $x^{\prime}$  be a vertex of $G$ before $x$ on $P_u$, see \Cref{pupv}.
    Obviously, any positive integer $l \geq r$ can be written as a linear combination of 2 and $r$ with non-negative integers as coefficients.
    Let $xC_ry$ be a shortest path between $x$ and $y$ on $C_r$.
    Suppose that $P_u x C_r y P_v$ is a walk between $u$ and $v$ of length $\ell(u,v)$.
    Let $\ell=\max_{u, v \in V(G)} \ell(u, v)$.
    We claim that $A^{\ell+r}>0$.
    For every $u$ and $v$, we show that there is a walk of length $\ell+r$ between $u$ and $v$. Because $\ell-\ell(u, v)+r \geq r$ and so
    $\ell -\ell(u, v)+r=2 a+r b$, where $a, b \geqslant 0$ are integers.
    Now, start at vertex $u$ and move along $P_u$ until $x$ and then repeat $xx^{\prime} x$, $a$ times and then move along $C_r, b$ times to arrive in $x$. Then move along $xC_ry$ and finally $y P_v v$, to obtain a walk of length $2 a+r b+\ell(u, v)=l+r$, between $u$ and $v$.
    This walk completes the proof.
    \begin{figure}[H]
        \centering
        \includegraphics[scale=0.9]{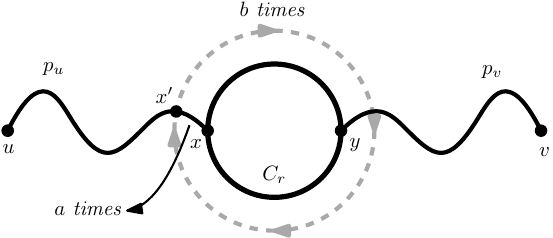}
        \caption{The picture used in the proof of \Cref{lem:A>0}.}
        \label{pupv}
    \end{figure}
    For the backward implication, we note that $\sum_{i=1}^n A^i>0$, where $|V(G)|=n$.
    So \Cref{connected} completes the proof.
\end{proof}

\begin{remark}\label{bipar_power}
By the same method one can show that if $G$ is a connected bipartite graph, with the adjacency matrix \[ A = \left[ \begin{array}{c|c} 0 & B \\ \hline\\[-12.5pt] B^{T} & 0 \end{array} \right] ,\]
then there exists an even integer $k$ such that $\left(B B^{T}\right)^k>0$, and there exists an odd integer $k$ such that $A^k=\left[\begin{array}{c|c}0 & C \\\hline\\[-12.5pt] C^{T} & 0\end{array}\right]$, where $C>0$.
The converse is also true.
\end{remark}

\begin{lemma}\label{no_zero_row}
Let $G$ be a graph with no isolated vertex, and $A$ be the adjacency matrix of $G$. Then for every positive integer $k, A^k$ has no zero row.
\end{lemma}

\begin{proof}
Since $A$ has no isolated vertex, so $Aj>0$. Thus $A^2 j=A\left(Aj\right)>0$. This implies that for every positive integer $k, A^k>0$. This means that $A^k$ has no zero row.
\end{proof}

\begin{remark}\label{matching}
Note that the disjoint union of two copies of a connected and non-bipartite graph $G$ with the adjacency matrix $A$ can be factored into a connected graph $H$ and a disconnected graph $K$, where  $K$ is a disjoint union of $n$ copies of the complete graph $K_2$ and $H$ has the following adjacency matrix.
$$\quad\left[\begin{array}{c|c}A & 0 \\ \hline 0 & A\end{array}\right]=\left[\begin{array}{c|c}0 & A \\ \hline A & 0\end{array}\right]\left[\begin{array}{c|c}0 & I_n \\ \hline I_n & 0\end{array}\right].$$
Since $G$ is connected and non-bipartite, \Cref{connected} implies that there exists a $k$ such that $A^k>0$.
So we can assume there are an odd integer $k_1$ and an even integer $k_2$ such that $A^{k_1}>0$ and $A^{k_2}>0$.
Now one can see that $(AA^T)^{k_2}=A^{2k_2}>0$.
So we have shown that there is an even integer $k_2$ such that $(AA^T)^{k_2}>0$.
Next, we consider the $k_1$-th power of the adjacency matrix of $H$:
$$\left[\begin{array}{c|c} 0 & A^{k_1} \\ \hline A^{k_1} & 0\end{array}\right]$$
We note that $A^{k_1}>0$.
Now it follows from \Cref{bipar_power} that $H$ is connected.
\end{remark}

\begin{thm}\label{connectedness}
Let a graph $G$ be factored into a connected graph $H$ and a graph $K$ with no isolated vertex.
Then the following holds:
\end{thm}

\begin{enumerate}[label=\rm(\roman*)]
    \item  If $H$ is non-bipartite, then $G$ is connected.
    \item  If $H$ is bipartite and $G$ is not connected, then $K$ is a regular bipartite graph and consequently $n$ is even. Moreover, $G$ is a disjoint union of two non-bipartite connected graphs.
\end{enumerate}

\begin{proof}
Let $A, B$, and $C$ be the adjacency matrices of the graphs $G, H$, and $K$, respectively.
Then we have $A=BC$.
If $H$ is non-bipartite, then by \Cref{lem:A>0}, there exists  positive integer $k$ such that $B^k>0$.
Now, by \Cref{no_zero_row} we know that $A^k=B^k C^k>0$, and so it follows from \Cref{connected} that $G$ is connected, as desired.

Next, assume that $H$ is bipartite and $G$ is not connected.
By \Cref{deg}, we know that $K$ is a $k$-regular graph.
We have the following:
\begin{align*}
    B=\left[\begin{array}{cc}0 & B_{12} \\ B_{12}^\top & 0\end{array}\right] \textit{ and } v=\begin{bmatrix}
        X\\Y
    \end{bmatrix}>0,
\end{align*}
where $v$ is the Perron vector corresponding to $\lambda_{\max }(H)$.
Since $A, B, C$ mutually commute, and they are symmetric, they are simultaneously diagonalizable.
Since $H$ is connected $\lambda_{\max}(H)$ is simple and so $v$ is an eigenvector corresponding to $\lambda_{\text {max}}(G)$ and moreover we know that  $k=\lambda_{\text {max}}(K)$.
Hence we have $Av=\lambda_{\max }(G)v$.
Since $G$ is not connected, we conclude that the multiplicity of $\lambda_{\text {max }}(G)$ is at least $2$.
It follows from the proof of \Cref{eigen_max} that $\lambda_{\text {max}}(G)=\lambda_{\text {max}}(H) \lambda_{\text {max}}(K)$.
It is not hard to see that $\left[\begin{array}{c}X \\ -Y\end{array}\right]$ is an eigenvector corresponding to the eigenvalue $-\lambda_{\text {max }}(H)$ for $H$.
By \Cref{bipartite}(4), we know that the multiplicity $-\lambda_{\max }(H)$ is 1.
Since the multiplicity $-\lambda_{\max }(H)$ is 1, and the multiplicity $\lambda_{\max}(G)$ is at least $2$,
we infer that $-k$ is also an eigenvalue of $K$.
It is easy to see that $\left[\begin{array}{c}X \\ -Y\end{array}\right]$ is an eigenvector corresponding to the eigenvalue $-k$ of the graph $K$.

Let $C=\left[\begin{array}{ll}C_{11} & C_{12} \\ C_{12}^\top & C_{22}\end{array}\right]$ be the adjacency matrix of $K$, where $Q $ is an $|X| \times |Y|$-matrix.
Then the following hold:
$$
\left[\begin{array}{cc}
C_{11} & C_{12} \\ C_{12}^\top & C_{22}
\end{array}\right]\left[\begin{array}{l}
X \\
Y
\end{array}\right]=k\left[\begin{array}{l}
X \\
Y
\end{array}\right],\left[\begin{array}{ll}
C_{11} & C_{12} \\ C_{12}^\top & C_{22}
\end{array}\right]\left[\begin{array}{c}
X \\
Y
\end{array}\right]=-k\left[\begin{array}{c}
X \\
-Y
\end{array}\right]
$$
\noindent Thus we find that $C_{11} X \pm C_{12} Y= \pm k X, C_{12}^\top X \pm C_{22} Y= kY$.
These imply that $C_{11} X=0$ and $C_{22}Y=0$.
Since $X>0$ and $Y>0$, we have $C_{11}=0$ and $C_{22}=0$.
So $K$ is bipartite.
If $|V(G)|=n$, then since $K$ is regular bipartite, we have $n$ is even.

Thus we have shown that $C=\left[\begin{array}{ll}0 & C_{12} \\ C_{12}^\top & 0\end{array}\right]$ and so $A=\left[\begin{array}{cc}B_{12}C_{12}^\top & 0 \\ 0 & B_{12}^\top C_{12}\end{array}\right]$.
This shows that $G$ is a disjoint union of two connected graphs $G_1,G_2$.
Next, we show that $G_1,G_2$ are both connected.
By \Cref{bipar_power}, there exists a positive even integer $k$ such that $B^k=\left[\begin{array}{cc}D & 0 \\ 0 & E\end{array}\right]$ with $D>0$ and $E>0$. So we have $A^k=B^k C^k$ and since $K$ has no isolated vertex by \Cref{no_zero_row}, we conclude that $A^k=\left[\begin{array}{cc}M & 0 \\  0 & N\end{array}\right]$, where $M>0, N>0$. So $G_1,G_2$ are both non-bipartite and connected.
\end{proof}

\noindent Next, it follows from \Cref{connectedness} the following Corollaries:

\begin{corollary}
Let a graph $G$ of order $n$ be factored into a connected graph $H$ and a graph $K$ with no isolated vertex.
If $n$ is odd, then $G$ is connected.
\end{corollary}

\begin{corollary}
Let a disconnected graph $G$ be factored into a connected graph $H$ and a graph $K$ with no isolated vertex. Then $H$ and $K$ are both bipartite.
\end{corollary}

\begin{problem}
    In Part (ii) of \Cref{connectedness}, is it true that two connected components of $G$ are isomorphic? 
\end{problem}

\section{Factorizations of trees and forests}\label{tree_sec}

\begin{lemma}\label{C4free}
Let a graph $G$ of order $n$ be factored into $H$ and $K$.
If $G$ contains neither $C_4$ nor isolated vertices, then $n$ is even.
\end{lemma}

\begin{proof}
We first assert that for each vertex $v$, either $\dd_H(v) \le 1$ or $\dd_K(v) \le 1$.
Assume to the contrary there exists a vertex $v$ with $\dd_H(v)\ge 2$ and  $\dd_K(v)\ge 2$.
Then $v$ has two neighbors $u_1,u_2$ in $H$ and two neighbors $w_1,w_2$ in $K$.
By definition, $G$ contains a cycle $C_4$ with edges $\{u_1,w_1\}, \{w_1,u_2\}, \{u_2,w_2\}, \{w_2,u_1\}$; a contradiction to the assumption on $G$.
By \Cref{degree}, 
$\dd_G(v)=\dd_H(v) \cdot \dd_K(v)$ for each vertex $v$.
As $G$ contains no isolated vertices,  $\dd_G(v) \ge 1$, and hence either $\dd_H(v) = 1$ or $\dd_K(v) = 1$ for each vertex $v$.
So,  for each vertex $v$, 
$$\dd_G(v)=\dd_H(v) + \dd_K(v) -1.$$
Now summing the above equality over all vertices, we have 
$$2|E(G)|=2|E(H)|+2|E(K)|-n,$$
which implies that $n$ is even.
\end{proof}

\begin{thm}
Any tree of order at least $2$ is not factorizable.
\end{thm}

\begin{proof}
Assume that there exists a tree $T$ of order $n \ge 2$ that can be factored into two graphs.
By \Cref{C4free}, $n$ is even, and then $T$ has $n-1$ (an odd number of) edges.
However, by \cite[Theorem 4]{Manjunatha}, if a graph $G$ has a factorization, then it must have an even number of edges, which yields a contradiction.
\end{proof}

\begin{thm}
Any forest with no isolated vertices and with an odd number of connected components is not factorizable.
\end{thm}

\begin{proof}
Let $F$ be a forest of order $n$ with no isolated vertices and $k$ connected components, where $k$ is odd.
Then $n$ is even by \Cref{C4free}.
Note that $F$ has $n-k$ (an odd number of) edges.
So $F$ is not factorizable by \cite[Theorem 4]{Manjunatha}.
\end{proof}
\begin{problem}
Characterize all factorizable forests.
\end{problem}

\bibliography{references_LAA.bib}
\end{document}